\documentclass[a4paper]{article}
\usepackage{amsmath,amssymb,amsthm}
\usepackage[utf8]{inputenc}
\usepackage[T1]{fontenc}
\usepackage[english]{babel}
\usepackage{mathrsfs}
\usepackage[a4paper]{geometry}
\usepackage{longtable}
\usepackage{graphicx}
\usepackage{subcaption}
\usepackage{mathtools}
\usepackage{color}
\usepackage{enumitem}
\usepackage{url}
\usepackage{setspace}
\usepackage{eucal}
\usepackage{pst-node}
\usepackage{tikz-cd}

\newtheorem{thm}{Theorem}[section] 
\newtheorem{theorem}[thm]{Theorem}
\newtheorem{proposition}[thm]{Proposition}
\newtheorem{lemma}[thm]{Lemma}
\newtheorem{corollary}[thm]{Corollary}
\theoremstyle{definition}
\newtheorem{remark}[thm]{Remark}
\newcommand{\ie}{\emph{i.e.}}
\newcommand{\eg}{\emph{e.g.}}
\newcommand{\cf}{\emph{cf.}}

\newcommand{\Real}{\mathbb{R}}
\newcommand{\Com}{\mathbb{C}}

\newcommand{\eps}{\varepsilon}
\newcommand{\sii}{L^2}

\newcommand{\sgn}{\mathrm{sgn}}
\DeclareMathOperator{\dom}{dom}

\DeclareRobustCommand{\rchi}{{\mathpalette\irchi\relax}}
\newcommand{\irchi}[2]{\raisebox{\depth}{$#1\chi$}}
\newcommand\restr[2]{{
  \left.\kern-\nulldelimiterspace 
  #1 
  \littletaller 
  \right|_{#2} 
  }}

\newcommand{\littletaller}{\mathchoice{\vphantom{\big|}}{}{}{}}
\usepackage[normalem]{ulem}
\definecolor{DarkGreen}{rgb}{0,0.5,0.1} 
\definecolor{DarkBlue}{rgb}{0,0.1,0.7} 

\newcommand\soutM{\bgroup\markoverwith
{\textcolor{DarkBlue}{\rule[.5ex]{2pt}{1pt}}}\ULon}
\newcommand{\Hm}[1]{\leavevmode{\marginpar{\tiny%
$\hbox to 0mm{\hspace*{-0.5mm}$\leftarrow$\hss}%
\vcenter{\vrule depth 0.1mm height 0.1mm width \the\marginparwidth}%
\hbox to
0mm{\hss$\rightarrow$\hspace*{-0.5mm}}$\\\relax\raggedright #1}}}

\begin{document}

\title{\textbf{Dirac operators on the half-line:
stability of spectrum and non-relativistic limit}}
\author{David Kram\'ar \ and \ David Krej\v ci\v r\'ik}
\date{\small 
\emph{
\vspace{-4ex}
\begin{quote}
\begin{center}
Department of Mathematics, Faculty of Nuclear Sciences and 
Physical Engineering, Czech Technical University in Prague, 
Trojanova 13, 12000 Prague 2, Czech Republic 
\\
kramada1@fjfi.cvut.cz, david.krejcirik@fjfi.cvut.cz
\end{center}
\end{quote}
}
\smallskip 
16 May 2024
}
\maketitle

\begin{abstract}
We consider Dirac operators on the half-line,
subject to generalised infinite-mass boundary conditions.
We derive sufficient conditions which guarantee 
the stability of the spectrum against 
possibly non-self-adjoint potential perturbations 
and study the optimality of the obtained results.
Finally, we establish a non-relativistic limit 
which makes a relationship of the present model
to the Robin Laplacian on the half-line.
%
%
\end{abstract}

\section{Introduction}
%
According to classical physics, 
the electrons in atoms would have collapsed into the nucleus 
in a matter of nanoseconds~\cite{1915}, 
which is in direct contradiction to our experience. 
This paradox was one of the impetuses leading 
to the development of quantum mechanics
which particularly explains the stability of matter~\cite{stability}. 
Mathematically, this is a consequence of the stability
of the spectrum of the Laplacian~$-\Delta$ in~$\Real^3$
against additive perturbations $V:\Real^3 \to \Com$
satisfying the subordination condition
\cite{katoBS,FKV,Birman}
\begin{equation}\label{sub}
  \forall \psi \in W^{1,2}(\Real^3)
  \,, \qquad
  \int_{\Real^3} |V|\,|\psi|^2 \leq a \int_{\Real^3} |\nabla\psi|^2
  \,,
\end{equation}
with $a < 1$.
Indeed, \eqref{sub}~together with the Hardy inequality implies 
that the spectrum of the Coulomb Hamiltonian is bounded from below. 
The strength of the condition~\eqref{sub} is that it ensures
not only the set stability $\sigma(-\Delta+V) = \sigma(-\Delta)$,
but also the absence of embedded eigenvalues.
Moreover, the fact that~$V$ is allowed to be complex-valued
makes the spectral analysis substantially more challenging,
while the problem is still important for applications~\cite{KSTV}. 

There exist generalisations of the subordination condition~\eqref{sub}
to higher dimensions \cite{katoBS,Frank,FKV,Birman}.
Because of the absence of Hardy inequality,
the condition is void in dimensions~$1$ and~$2$. 
However, there exist two-dimensional analogues 
for the magnetic Laplacian \cite{FKV2,CFK}.
To get a non-trivial result in dimension one,
it is necessary to restrict the Laplacian to a half-line,
subject to repulsive boundary conditions of Robin type
\cite{CK2,KLS}.

The motivation of the present paper is to establish 
one-dimensional analogues of~\eqref{sub} 
in the relativistic setting of Dirac operators on the half-line.
While this model was considered previously
\cite{Cuenin_2014,Enblom_2018},
the stability of the spectrum has not been addressed
and our objective is to fill in this gap. 
For an analogous stability problem for 
the Dirac operator in higher dimensions, 
see \cite{AFKS,CFK,Ancona-Fanelli-Schiavone_2022,Mizutani-Schiavone_2022}.

To be more specific, given $m \geq 0$ (mass),
we consider the unperturbed Dirac operator
\begin{equation}\label{operator}
D_0 :=  
\begin{pmatrix}
m & -\dfrac{d}{dx} \\
\dfrac{d}{dx} & -m \\
\end{pmatrix}
\qquad \mbox{in} \qquad 
\sii\big((0,\infty),\Com^2\big)
\end{equation}
with the domain consisting of functions
$
  \psi = 
  \begin{psmallmatrix}
    \psi_1 \\ \psi_2
  \end{psmallmatrix}
  \in H^1\big((0,\infty),\Com^2\big)
$
satisfying the one-pa\-ram\-e\-tric class of
generalised infinite-mass boundary conditions
\begin{equation}\label{bc.intro}
  \psi_1(0)\cot(\alpha) = \psi_2(0)
  \qquad \mbox{with} \qquad
  \alpha \in \left(0,\mbox{$\frac{\pi}{2}$}\right) \,.
\end{equation}
The latter restriction ensures that~$D_0$
is self-adjoint and its spectrum is purely continuous: 
\begin{equation}\label{spectrum}
  \sigma(D_0)
  = \sigma_\mathrm{c}(D_0)
  = (-\infty,-m] \cup [m,+\infty)
  \,.
\end{equation}
Given $V: (0,\infty) \to \Com^{2,2}$, 
we consider the perturbed operator $D_0 + V$.
To be more specific, 
since we merely assume 
$V \in L^1\big( (0,\infty), \Com^{2,2} \big)$,
we introduce a closed (not necessarily self-adjoint) extension~$D_V$  
of the operator sum $D_0 + V$
via Kato's resolvent formula
(see~\cite{katoBS,Cuenin-Laptev-Tretter_2014} and below).
For almost every $x \in (0,\infty)$, $V(x)$ is a $2 \times 2$ matrix
and we denote by $|V(x)|$ its operator norm in~$\Com^2$.
Our main result reads as follows.
\begin{theorem}\label{Thm.main}
Let $V \in L^1\big((0,\infty),\Com^{2,2};(1+x) \, dx\big)$ satisfy
\begin{align}\label{sufcon}
\int_0^\infty \int_0^\infty 
\lvert V(x) \rvert 
\left[ 1 + \big(q + 2m \min(x,y)\big)^2\right]  
\lvert V(y) \rvert \, dx \, dy < 1 \,,
\end{align}
where $q := \max\big(\cot(\alpha),{\cot(\alpha)}^{-1}\big)$.
Then  
$
  \sigma(D_V)
  = \sigma_\mathrm{c}(D_V)
  = \sigma(D_0)
$.  
\end{theorem}

The sufficient condition~\eqref{sufcon} follows 
as a consequence of a better (but less explicit) result 
which is of the same nature as~\eqref{sub}
(see Theorem~\ref{Thm.central} and Remark~\ref{Rem.sub}). 
Moreover, the spectral stability follows as a consequence
of the fact that $D_V$ 
is similar to $D_0$ 
via bounded and boundedly invertible similarity transformation
(see~Theorem~\ref{Thm.central} and Corollary~\ref{Corol.central}).
Consequently, the possibly non-self-adjoint 
operator~$D_V$ is quasi-self-adjoint
\cite{KSTV,KS-book}.

We have not been able to prove the optimality of~\eqref{sufcon}.
To this purpose, we also establish an alternative 
sufficient condition for special potentials 
having~$V_{11}$ as the only non-zero component
and this turns out to be optimal. 

Our next goal is to relate Theorem~\ref{Thm.main} 
to the non-relativistic analogue given by 
the self-adjoint Laplacian (here $m>0$)
\begin{equation}\label{operator.non}
  S_0 := -\frac{1}{2m} \frac{d^2}{dx^2}
  \qquad \mbox{in} \qquad
  L^2((0,\infty))
\end{equation}
with the domain consisting of functions 
$\psi \in H^2((0,\infty))$ satisfying 
the repulsive Robin boundary conditions 
\begin{equation}\label{bc.non}
  \psi^\prime(0) = \beta \psi(0)
  \qquad \mbox{with} \qquad
  \beta \in \left(0,+\infty\right) \,.
\end{equation}
The Dirichlet problem formally corresponds to $\beta=+\infty$.
Given $V:(0,\infty) \to \Com$ with $V \in L^1((0,\infty))$, 
the symbol~$S_V$ denotes the m-sectorial extension 
of the sum $S_0+V$ obtained via the form sum.
Spectral enclosures for $S_V$ were established already in
\cite{Frank-Laptev-Seiringer_2011,Enblom_2017}.
However, the stability of the spectrum was
not considered in these references 
and related results can be found only 
in more recent works~\cite{CK2,KLS}.

\begin{theorem}[\cite{KLS}]\label{Thm.KLS} 
Let $V \in L^1\big((0,\infty),\Com;(1+x) \, dx\big)$ satisfy
\begin{equation}\label{sufcon.non}
\int_0^\infty \int_0^\infty \lvert V(x) \rvert  
\left(\frac{2m}{\beta} + 2m \min(x,y)\right)^2 \lvert V(y) \rvert \, dx \, dy < 1 \,.
\end{equation}
Then 
$
  \sigma(S_V) 
  = \sigma_\mathrm{c}(S_V) = \sigma(S_0)
$.
\end{theorem}

In this paper, we explain 
how to understand~\eqref{sufcon.non} 
as the non-relativistic limit of~\eqref{sufcon}.
We consistently introduce the speed of light~$c$
into the action of~$D_0$ 
as well as to the boundary condition~\eqref{bc.intro}
and obtain~$S_0$ as the limit 
(in a norm resolvent sense)
of a suitably renormalised~$D_0$   
after sending~$c$ to~$+\infty$.  
It turns out that the right correspondence 
between the relativistic and non-relativistic 
boundary conditions is given by
\begin{equation}\label{identification}
  \beta = 2 \cot\alpha \,.
\end{equation}
Then \eqref{sufcon.non}~is a limit of 
a $c$-dependent version of~\eqref{sufcon} as $c \to +\infty$.
From this perspective, the infinite-mass (also known as MIT)
boundary condition given by $\alpha = \frac{\pi}{4}$ 
is a relativistic version of the Robin parameter $\beta=2$.

We stress that our non-relativistic limit 
is in no mathematical contradiction with the other approaches 
\cite{Cuenin_2014,Arrizibalaga-LeTreust-Raymond_2017,
Behrndt-Frymark-Holzmann-Stelzer-Landauer},
where the (generalised) infinite-mass boundary condition 
is interpreted as a relativistic version of the Dirichlet 
boundary condition 
(\ie, $\beta=+\infty$, irrespectively of the value of~$\alpha$). 
Indeed, the argument of the other authors
is based on a different introduction 
of the speed of light into the operator.
In particular, the boundary condition is left $c$-independent
and that is why the dependence on~$\alpha$ is lost in the limit, 
see \cite[Prop.~3.1]{Cuenin_2014}. 
Our approach is based on physics arguments~\cite{Thaller}
recently adopted in a related problem in~\cite{lukas}.

The paper is organised as follows. 
In Section~\ref{Sec.pre} we write down the integral kernel
of the resolvent of the unperturbed operator~$D_0$
and establish two crucial uniform estimates 
(Lemmata~\ref{supremum} and~\ref{supremum.bis}). 
In Section~\ref{Sec.stability} we introduce the perturbed 
operator~$D_V$ and establish various sufficient conditions
which guarantee the similarity of~$D_V$ to~$D_0$.
In particular, Theorem~\ref{Thm.main} 
(based on Lemma~\ref{supremum}) is proved,
but we also establish an alternative result,
Theorem~\ref{Thm.alt} (based on Lemma~\ref{supremum.bis}). 
In Section~\ref{Sec.optimal} we prove the optimality 
of Theorem~\ref{Thm.alt} employing Dirac delta potentials.
In Section~\ref{Sec.non}, 
introducing the speed of light~$c$ in a refined way, 
we prove a convergence of~$D_0$ to~$S_0$ in a norm-resolvent sense 
as $c \to +\infty$. In particular, the compatibility of 
Theorems~\ref{Thm.main} and~\ref{Thm.KLS} is established.

\section{The free Dirac operator and its resolvent}\label{Sec.pre}
%
From now on, we abbreviate $\mathbb{R}_+ := (0,\infty)$
and consider the Hilbert space 
$L^2\big(\mathbb{R}_+,\mathbb{C}^2\big)$.
The norm and inner product in this space 
(and other functional spaces) is denoted 
by $\| \cdot \|$ and $\left(\cdot,\cdot\right)$, respectively.
The latter is assumed to be antilinear in its first argument.
Recall that the absolute value notation $|\cdot|$ is maintained 
for the vector and operator norms of~$\Com^2$.

The half-line Dirac operators~\eqref{operator}, 
subject to the generalised infinite-mass boundary conditions~\eqref{bc.intro}, 
have already been considered in \cite{Cuenin_2014,Enblom_2018}. 
We particularly follow the notation of  
the pioneering work of Cuenin~\cite{Cuenin_2014},
from where we overtake the following formula for the resolvent. 
\begin{proposition}
Let $\alpha \in (0,\frac{\pi}{2})$. 
For $z \in \rho(D_0)$ the resolvent 
$(D_0 -z)^{-1}$ acts as an integral operator 
with the integral kernel
\begin{align}\label{resolvent}
\mathcal{R}_\alpha(x,y;z) = \frac{1}{W}\left[
\psi_\alpha(x;z) \, \phi_\alpha(y;z)^\top \, \Theta(x-y) 
+ \phi_\alpha(x;z) \, \psi_\alpha(y;z)^\top \, \Theta(y-x)
\right],
\end{align}
where~$\Theta$ stands for the Heaviside function and
\begin{align*}
\psi_\alpha(x;z) &:= \exp(ik(z)x)\begin{pmatrix} i\zeta(z) \\ -1 \end{pmatrix}, \\
\phi_\alpha(y;z) &:= \begin{pmatrix} \cos(k(z)y) + \zeta(z) \cot(\alpha) \sin(k(z)y) \\ -\zeta(z)^{-1} \sin(k(z)y) + \cot(\alpha) \cos(k(z)y) \end{pmatrix},
\end{align*}
with coefficients
\begin{align*}
 k(z) := \sqrt{z^2 - m^2}, \qquad \zeta(z) := \frac{z + m}{k(z)}, \qquad W := 1 + i\zeta(z)\cot(\alpha).
\end{align*}
The square root is chosen so that $\Im[k(z)] > 0$.
\end{proposition}
\begin{proof}
For the reader's convenience, we give a short proof 
of the resolvent formula. 
Moreover, the proof provides an insight into the structure,
which will be useful when deriving 
the non-relativistic limit in Section~\ref{Sec.non}.

Let us consider the whole-line Dirac operator~$D$ 
in $L^2(\mathbb{R}, \mathbb{C}^2)$,
which acts as~$D_0$ in~\eqref{operator},
subject to continuity boundary conditions at zero,
namely the domain of~$D$ is $H^1(\mathbb{R}, \mathbb{C}^2)$.  
Using the traditional trick 
$
  (D - z)^{-1} 
  = (D + z) (D^2 - z^2)^{-1}
$
together with the fact that the resolvent kernel 
of the Schr\"odinger operator
$D^2 = (-\Delta + m^2) I_{\Com^2}$ is well known, 
one easily verifies that the integral kernel 
of $(D - z)^{-1}$ reads 
\begin{align}\label{dirac resolvent}
\mathcal{R}(x,y;z) = \frac{1}{2} 
\begin{pmatrix}
i\zeta(z) & \sgn(x-y)\\
-\sgn(x-y) & i\zeta^{-1}(z)
\end{pmatrix}
\exp(ik(z)\lvert x-y \rvert)
\,.
\end{align}

Let $\phi \in L^2(\mathbb{R}_+,\mathbb{C}^2)$ be a function on the half-line and $\mathbb{A} := \mbox{diag}(\mu_1,\mu_2) \in \mathbb{C}^{2,2}$ be a diagonal matrix. 
For every $x \in \Real$,
we define $\phi_\mathbb{A}(x) := \phi(\lvert x \rvert)\Theta(x) +\mathbb{A} \phi(\lvert x \rvert)\Theta(-x)$,
which is a function in $L^2(\mathbb{R},\mathbb{C}^2)$
but not necessarily in $H^1(\mathbb{R},\mathbb{C}^2)$. 
Then 
\begin{align*}
\left((D-z)^{-1}\phi_\mathbb{A}\right)(x) 
&= \int_\mathbb{R} \mathcal{R}(x,y;z)\phi_\mathbb{A}(y) dy 
\\
&= \int_0^{\infty} ( \mathcal{R}(x,y;z) 
+ \mathcal{R}(x,-y;z)\mathbb{A})~\phi(y)dy 
=: 
\begin{pmatrix}
\xi_1(x)\\
\xi_2(x)
\end{pmatrix}.
\end{align*}
Demanding the boundary condition 
\begin{align*}
\xi_1(0)\cot(\alpha) = \xi_2(0) 
\end{align*}
yields the equation
\begin{align*}
\int_0^{\infty} [&\left(\mathcal{R}_{11}(0,y;z)
+ \mu_1 \mathcal{R}_{11}(0,-y;z)\right)\cot(\alpha) 
- \mathcal{R}_{21}(0,y;z) 
- \mu_1 \mathcal{R}_{21}(0,-y;z)]~\phi_1(y) dy \\
= \int_0^{\infty} [&\mathcal{R}_{22}(0,y;z)
+ \mu_2 \mathcal{R}_{22}(0,-y;z) 
- \left(\mathcal{R}_{12}(0,y;z) 
+ \mu_2 \mathcal{R}_{12}(0,-y;z) \right)\cot(\alpha) ]~\phi_2(y) dy .
\end{align*}
Since this equation has to be satisfied 
for every $\phi \in L^2(\mathbb{R}_+,\mathbb{C}^2)$, 
for the special choice $\phi_1 = 0$, respectively $\phi_2 = 0$, 
it breaks down into two independent conditions
\begin{align*}
\left(\mathcal{R}_{11}(0,y;z)+ \mu_1 \mathcal{R}_{11}(0,-y;z)\right)\cot(\alpha) &= \mathcal{R}_{21}(0,y;z) + \mu_1 \mathcal{R}_{21}(0,-y;z),\\
\left(\mathcal{R}_{12}(0,y;z) + \mu_2 \mathcal{R}_{12}(0,-y;z) \right)\cot(\alpha) &= \mathcal{R}_{22}(0,y;z)+ \mu_2 \mathcal{R}_{22}(0,-y;z),
\end{align*}
to be satisfied for all $y \in \mathbb{R}_+$. 
From the expression \eqref{dirac resolvent} we obtain a unique solution 
for $\mu_1,\mu_2$ leading to
$$
  \mathbb{A} =
  \sigma_3 \eta(\alpha)
  \qquad \mbox{with} \qquad
  \eta(\alpha) :=
  \frac{1 -i \zeta(z)\cot(\alpha)}{1 + i \zeta(z)\cot(\alpha)}
  \qquad \mbox{and} \qquad
  \sigma_3 :=
  \begin{pmatrix}
    1 & 0 \\
    0 & -1
  \end{pmatrix}
  \,.
$$
 
The integral kernel of the resolvent 
$\left(D_0 - z \right)^{-1}$ therefore reads
\begin{align}\label{resolvent-matrix form}
\mathcal{R}_\alpha(x,y;z) = \mathcal{R}(x,y;z) + \mathcal{R}(x,-y;z)\sigma_3\eta(\alpha),
\end{align}
for all $x,y \in \mathbb{R}_+$. 
It is straightforward to check that 
\begin{align*}
\mathcal{R}_\alpha(x,y;z) = R_\alpha(x,y;z) \, \Theta (x-y) 
+ R_\alpha(y,x;z)^\top \, \Theta(y-x)
\end{align*}
with
\begin{align*}
[R_\alpha(x,y;z)]_{11} &= \frac{i\zeta(z)}{1 + i\zeta(z)\cot(\alpha)}\exp(ik(z)x) \left[ \cos(k(z)y) + \zeta(z)\cot(\alpha)\sin(k(z)y) \right], \\ 
[R_\alpha(x,y;z)]_{12} &= \frac{i\zeta(z)}{1 + i\zeta(z)\cot(\alpha)}\exp(ik(z)x)\left[\cot(\alpha)\cos(k(z)y) - \zeta(z)^{-1}\sin(k(z)y) \right],\\ 
[R_\alpha(x,y;z)]_{21} &= \frac{-1}{1 + i\zeta(z)\cot(\alpha)}\exp(ik(z)x)\left[ \cos(k(z)y) + \zeta(z)\cot(\alpha)\sin(k(z)y) \right],\\ 
[R_\alpha(x,y;z)]_{22} &= \frac{-1}{1 + i\zeta(z)\cot(\alpha)}\exp(ik(z)x)\left[\cot(\alpha)\cos(k(z)y) - \zeta(z)^{-1}\sin(k(z)y) \right].
\end{align*}
This coincides with~\eqref{resolvent}.
\end{proof}

The following lemma is the main technical result of this paper.
Here it is fundamental that the boundary parameter~$\alpha$ 
is assumed to belong to the interval $(0,\frac{\pi}{2})$, 
\cf~\eqref{bc.intro}.
Indeed, $D_0$ possesses a discrete eigenvalue
in the interval $(-m,m)$ if $\alpha \in (\frac{\pi}{2},\pi)$
and $\mathcal{R}_\alpha(x,y;z)$
has a non-removable singularity at $z = \pm m$ 
even if $\alpha \in \{0,\frac{\pi}{2}\}$
(zigzag boundary conditions). 
Recall the notation~$q$ introduced in Theorem~\ref{Thm.main}.

\begin{lemma}\label{supremum}
Let $\alpha \in (0,\frac{\pi}{2})$. Then
\begin{align*}
\sup_{z\in\rho(D_0)} \lvert \mathcal{R}_\alpha(x,y;z) \rvert^2 
= 1 + \big(q + 2m \min(x,y)\big)^2 \,.
\end{align*}
\end{lemma}
\begin{proof}
The norm of the resolvent kernel $\mathcal{R}_\alpha(x,y;z)$ as an operator $\mathbb{C}^2 \rightarrow \mathbb{C}^2$ reads
\begin{align*}
\lvert \mathcal{R}_\alpha(x,y;z) \rvert = \frac{1}{\lvert W(z) \rvert} \left[ \lvert \psi_\alpha(x;z)\rvert \lvert \phi_\alpha(y;z)\rvert \Theta(x-y) + \lvert \psi_\alpha(y;z)\rvert \lvert \phi_\alpha(x;z)\rvert \Theta(y-x) \right].
\end{align*}
Without loss of generality, we shall further consider only the case $x > y$. 
To simplify further expressions, let us introduce the notation
\begin{align*}
\eta_1(z) &:= \frac{1}{4} \left(1 + \lvert \zeta(z) \rvert^{-2}\right) \lvert 1 - i\zeta(z)\cot(\alpha)\rvert^2 ,\\
\eta_2(z) &:= \frac{1}{4} \left(1 + \lvert \zeta(z) \rvert^{-2}\right) \lvert 1 + i\zeta(z)\cot(\alpha)\rvert^2 ,\\
\eta_3(z) &:= \frac{1}{2}\left(1 - \cot^2(\alpha)\lvert \zeta(z) \rvert^2\right),\\
\eta_4(z) &:= \frac{1}{2}\left(2\Re[\zeta(z)]\cot(\alpha)\right).
\end{align*}
A straightforward calculation yields
\begin{align*}
\lvert \psi_\alpha(x;z) \rvert^2 = \ & \exp(-2\Im[k(z)]x) \left(\lvert\zeta(z)\rvert^2 + 1\right), \\
\lvert \phi_\alpha(y;z) \rvert^2 =\ & \eta_1(z) \exp(-2\Im[k(z)]y) + \eta_2(z) \exp(2\Im[k(z)]y) \\
&+ \left(1 - \lvert \zeta(z) \rvert^{-2}\right) \left(\eta_3(z) \cos(2\Re[k(z)]y) + \eta_4(z) \sin(2\Re[k(z)]y) \right).
\end{align*}

Since the sums, products, and compositions of holomorphic functions are holomorphic, it is obvious that the resolved kernel $\mathcal{R}_\alpha(x,y;z)$ is a holomorphic function of the spectral parameter~$z$. Therefore, the supremum of its modulus cannot be achieved in the resolvent set. This fact follows from the maximum modulus principle stated in~\cite{max} for general Banach space-valued maps. The fact that $\mathbb{C}^2$ equipped with the standard norm satisfies the appropriate assumptions follows from~\cite{convex}.

On the other hand, the supremum itself has to exist. Therefore, either the supremum lies at complex infinity or it is achieved somewhere 
in the spectrum~\eqref{spectrum}. 
We shall show that the latter is true. First, we show that $\lvert \mathcal{R}_\alpha(x,y;z) \rvert$ can be uniformly bounded as $z \rightarrow \infty$. 

Estimating the exponential functions one has
\begin{align*}
\lim \limits_{z \to \infty} \lvert \mathcal{R}_\alpha(x,y;z) \rvert^2 \Theta(x-y) = \ & \lim \limits_{z \to \infty} \frac{1}{\lvert W \rvert^2} \lvert \psi_\alpha(x;z)\rvert^2 \lvert \phi_\alpha(y;z)\rvert^2 \Theta(x-y)\\ 
\leq \ &  \lim \limits_{z \to \infty} \frac{\left(1+ \lvert \zeta(z) \rvert^2 
\right)}{\lvert W \rvert^2} \left(\eta_1(z) + \eta_2(z) \right) \\
&+ \lim \limits_{z \to \infty} \frac{\left(1 - \lvert \zeta(z) \rvert^{-2}\right)}{\lvert W \rvert^2} \left(\cos(2\Re[k(z)]y) \eta_3(z) + \sin(2\Re[k(z)]y) \eta_4(z)\right) \\
= \ & 2
\end{align*}
The last equality follows from the facts that 
$
  \lvert 1 \pm i\zeta(z)\cot(\alpha)\rvert^2 
  \to 1 + \cot^2(\alpha)
$ 
and 
$
   \lvert \zeta(z)\rvert^2 \to 1
$
as $z \to \infty$,
which is a consequence of the asymptotics of $k(z)$ and $\zeta(z)$. 
Indeed, it can be easily checked that, 
under the notation $z = z_1 + i z_2$ with $z_1,z_2 \in \Real$, 
it holds that
\begin{align*}
k(z) = \frac{z_1 z_2}{a(z)} + i a(z),    
\end{align*}
with
\begin{align*}
    a(z) := \frac{1}{\sqrt{2}} \left( z_2^2 - z_1^2 + m^2 + \sqrt{( z_2^2 - z_1^2 + m^2)^2 + 4 z_1^2 z_2^2} \right)^{1/2}.
\end{align*}
Furthermore, for $z_2 \neq 0$, by simple calculations one finds that, 
    \begin{itemize}
        \item $\lim \limits_{\lvert z_2 \rvert \to + \infty} a(z) = +\infty \quad \mbox{as} \quad o(\lvert z_2 \rvert)$, 
        \item $\lim \limits_{\lvert z_1 \rvert \to + \infty} a(z) = \lvert z_2 \rvert,$
        \item $\lim \limits_{\lvert z_1 \rvert \to + \infty} \Im(\zeta(z)) = \lim \limits_{\lvert z_2 \rvert \to + \infty} \Im(\zeta(z)) = 0.$
    \end{itemize}

Now, we investigate the behaviour of the restriction $\lvert \mathcal{R}_\alpha(x,y;z) \rvert^2$ to the spectrum. That is,
$u := z \in \sigma(D_0) = (-\infty,-m] \cup [m,+\infty)$.
The coefficients $k(u)$ and $\zeta(u)$ are now purely real, and the restriction of $\lvert \mathcal{R}_\alpha(x,y;z) \rvert^2$ reads
\begin{align*}
\lvert \mathcal{R}_\alpha(x,y;u) \rvert^2 \Theta(x-y) 
= \ &  \frac{2u^2}{u^2 - m^2} + \cos(2k(u)y)\frac{u-m-\cot^2(\alpha)(u+m)}{u-m + \cot^2(\alpha)(u+m)} \frac{2mu}{u^2 - m^2} \\
 & + \sin(2k(u)y)\frac{\cot(\alpha)4m}{\sqrt{u^2 - m^2}} \frac{1}{u-m + \cot^2(\alpha)(u+m)} =: \rchi(u,y).
\end{align*}
We claim that the function $\rchi(u,y)$ has no local extremes for $u \in \sigma(\mathcal{D}_\alpha)$. Indeed, the derivative
\begin{equation}
\begin{aligned}
\frac{\partial\rchi}{\partial u}(u,y) 
= \ & \xi_1(u) + \xi_2(u)\cos(2k(u)y) + \xi_3(u)\cos(2k(u)y)y \\
&+ \xi_4(u)\sin(2k(u)y) + \xi_5(u)\sin(2k(u)y)y 
\,,
\end{aligned}
\end{equation}
with some functions $\xi_1,\dots,\xi_5$ 
that we do not need to make explicit now,
can be viewed as the linear combination of linearly independent functions in variable $y$ for every fixed $u \in \sigma(D_0)$. 
Therefore, the only possible solution of the equation  
when the derivative is put equal to zero
is $\xi_i = 0$ for every $i \in \lbrace 1, \ldots, 5 \rbrace$. 
Since 
\begin{align*}
  \xi_1(u) = \frac{4mu}{u^2-m^2} \,,
\end{align*}
we conclude that $u = 0 \not \in (-\infty,-m] \cup [m,+\infty)$ 
if $m>0$ or that $\xi_1(u)=0$ cannot be satisfied if $m=0$.

Since we have established that $u \mapsto \rchi(u,y)$ has no local extremes,
its supremum must the maximum of the limits at the ``boundaries''. 
These are 
\begin{align*}
&\lim \limits_{u \to \pm \infty} \rchi(u,y) = 2,\\
&\lim \limits_{u \to -m} \rchi(u,y) 
= 1 + \big(\cot(\alpha) + 2my\big)^2, \\
&\lim \limits_{u \to m} \rchi(u,y) 
= 1 + \big({\cot(\alpha)}^{-1} + 2my\big)^2.
\end{align*}
Thus, we conclude that the claim of the lemma holds true.
\end{proof}

As an alternative to Lemma~\ref{supremum},
we shall use the following uniform bound
on the first component of the resolvent kernel.  
\begin{lemma}\label{supremum.bis}
Let $\alpha \in (0,\frac{\pi}{2})$. Then
\begin{align*} 
  \sup_{z\in\rho(D_0)} 
  \lvert \mathcal{R}_{\alpha,11}(x,y;z) \rvert 
  = \max\left(1, \frac{1}{\cot(\alpha)} + 2m \min(x,y)\right).
\end{align*}
\end{lemma}
\begin{proof}
It is not difficult to see that 
\begin{align*}
\lvert \mathcal{R}_{\alpha,11}(x,y;z) \rvert^2 \Theta(x-y) = \frac{\lvert \zeta(z) \rvert^2 \exp(-2\Im(k(z))x)}{\lvert 1 + i\zeta(z) \cot(\alpha) \rvert^2} \left( \Lambda_1(y;z) + \Lambda_2(y;z) \right), 
\end{align*}
where for $x > y > 0$ we have
\begin{align*}
\Lambda_1(x,y;z) &:= \frac{1}{4} \left[\exp(-2\Im(k(z))y) \lvert 1 - i\zeta(z)\cot(\alpha) \rvert^2 + \exp(2\Im(k(z))y) \lvert 1 + i\zeta(z)\cot(\alpha) \rvert^2 \right], \\
\Lambda_2(x,y;z) &:= \frac{1}{2} \left[\cos(2\Re(k(z))y) \left( 1 - 
\lvert \zeta(z) \rvert^2 \cot^2(\alpha) \right) + \sin(2\Re(k(z))y) 2\cot(\alpha) \Re(\zeta(z)) \right].
\end{align*}
Adopting the approach in the general setting of Lemma~\ref{supremum}, 
we see that
\begin{align*}
\sup_{z\in\rho(D_0)} \lvert \mathcal{R}_{\alpha,11}(x,y;z) \rvert = \max(l_\infty, l_{\pm \infty}, l_{\pm m}),
\end{align*}
where we denote
\begin{align*}
l_\infty &:= \limsup \limits_{z \to \infty} 
\, \lvert \mathcal{R}_{\alpha,11}(x,y;z) \rvert^2 \, \Theta(x-y), \\
l_{\pm \infty} &:= \limsup \limits_{z \to \pm \infty} \,
\lvert \mathcal{R}_{\alpha,11}(x,y;z) \rvert^2 \, \Theta(x-y), \\
l_{\pm m} &:= \lim \limits_{z \to \pm m} 
\lvert \mathcal{R}_{\alpha,11}(x,y;z) \rvert^2 \, \Theta(x-y).
\end{align*}

In the present case, 
the problem is more delicate, since the alternating part $\Lambda_2$ does not vanish as $z$ tends to a complex infinity. 
However, in general for any $\nu, \mu \in \mathbb{R}$ it holds that
\begin{align}\label{limsup}
\limsup \limits_{x \to +\infty} 
\left(\nu \cos(x) + \mu \sin(x) \right)
= \sqrt{\nu^2 + \mu^2}. 
\end{align}

\paragraph*{Calculation of $l_\infty$.}
The fact that $z$ tends to a complex infinity means that either $\lvert \Im(z) \rvert  \to +\infty$ or $\lvert \Re(z) \rvert  \to +\infty$. 
The only relevant case is the latter one, since $\lvert \mathcal{R}_{\alpha,11}(x,y;z) \rvert^2 \Theta(x-y) \to 0$ whenever $\lvert \Im(z) \rvert$ tends to $+\infty$, see the asymptotics of $k(z)$ and $\zeta(z)$ in the proof of 
Lemma~\ref{supremum}.
Combining \eqref{limsup} and the fact pointed out above, one can see that 
for any curve $\gamma : \mathbb{R}_+ \to \rho(D_0)$ such that $\lim \limits_{t \to +\infty} \lvert \Re(\gamma(t)) \rvert = +\infty$ we have
\begin{align*}
\limsup \limits_{t \to +\infty} \, \lvert \mathcal{R}_{\alpha,11}(x,y;\gamma(t)) \rvert^2 \Theta(x-y) \leq \frac{\exp(-2 \gamma_\infty^I  x)}{2} \left(\cosh(2 \gamma_\infty^I y) + 1 \right) \leq 1, 
\end{align*}
where $\gamma_\infty^I := \liminf \limits_{t \to \infty} \lvert \Im(\gamma(t)) \rvert $ which is, without loss of generality, assumed to be finite. Thus we conclude with $l_\infty \leq 1$.

\paragraph*{Calculation of $l_{\pm\infty}$.}
Restricted to the spectrum, \ie\ $z = u \in (-\infty, -m] \cup [m, +\infty)$, the absolute value of $\mathcal{R}_{\alpha,11}(x,y;u)$ reads
\begin{align*}
\lvert \mathcal{R}_{\alpha,11}(x,y;u) \rvert^2 \, \Theta(x-y) 
= \ & \frac{1}{2} \frac{u+m}{u-m} + \cos(2k(u)y) \frac{1}{2} \frac{u-m-\cot^2(\alpha)(u+m)}{u-m + \cot^2(\alpha)(u+m)} \frac{u+m}{u-m} 
\\
&+ \sin(2k(u)y)\frac{\cot(\alpha)(u+m)}{\sqrt{u^2 - m^2}} \frac{1}{u-m + \cot^2(\alpha)(u+m)}.    
\end{align*}
Once again, from \eqref{limsup} it follows that
\begin{align*}
l_{\pm \infty} = \limsup \limits_{u \to \pm \infty} \,
\lvert \mathcal{R}_\alpha(x,y;u) \rvert^2 \, \Theta(x-y) = 1.    
\end{align*}

\paragraph*{Limits at $\pm m$.}
By a direct calculation we obtain
\begin{align*}
l_{-m} &= \lim \limits_{u \to -m} \lvert \mathcal{R}_\alpha(x,y;u) \rvert^2
\, \Theta(x-y) = 0 \,, \\
l_{m} &= \lim \limits_{u \to m} \lvert \mathcal{R}_\alpha(x,y;u) \rvert^2
\, \Theta(x-y) = \left( \frac{1}{\cot(\alpha)} + 2my \right)^2
\,.
\end{align*}

The proof is completed by comparing the values 
of $l_\infty, l_{\pm\infty}, l_{\pm m}$.
\end{proof}
%

\section{Perturbations and the stability of the spectrum}\label{Sec.stability}
%
Given a potential $V:\Real_+ \to \Com^{2,2}$,
we would like to consider a closed realisation~$D_V$
of the perturbed operator $D_0 + V$.
Already the latter is a closed operator, 
with $\dom(D_0 + V) = \dom(D_0)$,
provided that~$V$ is bounded 
or, more generally, relatively bounded with respect to~$D_0$
with the relative bound less than~$1$.
This results in the requirement 
$V \in (L^2 + L^\infty)(\Real_+,\Com^{2,2})$.

To include a larger class of potentials,
it is possible to introduce~$D_V$
via the pseudo-Friedrichs extension~\cite{Birman}. 
This approach leads to 
$V \in (L^{1+\delta} + L^\infty)(\Real_+,\Com^{2,2})$
with any positive~$\delta$.
The advantage is a quite explicit form-wise interpretation
of~$D_V$ and the strong apparatus 
of the Birman--Schwinger principle developed in~\cite{Birman}. 
However, the desired setting $\delta=0$ does not seem to be reachable.

To reach the optimal scale $V \in (L^{1} + L^\infty)(\Real_+,\Com^{2,2})$,
it seems necessary to introduce~$D_V$  
via Kato's resolvent formula
\cite{katoBS,Cuenin-Laptev-Tretter_2014}.
We adopt this approach in this paper.
The abstract Kato's theorem, 
in a weaker form sufficient to our purposes,
reads as follows.

\begin{theorem}[{\cite{katoBS}}]\label{kato}
Let $H_0$ be a self-adjoint operator in a Hilbert space $\mathcal{H}$ and suppose that $A,B$ are closed operators in~$\mathcal{H}$ 
with $\dom(H_0) \subset \dom(A) \cap \dom(B)$
which are $H_0$-smooth and such that  
\begin{equation}\label{smooth}
        \sup \limits_{z \in \rho(H_0)} \lVert A(H_0 - z)^{-1}B^* \rVert 
        < 1 \,.
\end{equation}
Then there exists a closed extension $H_V$ of $H_0 + B^*A$ which is similar to $H_0$ satisfying for all $z \in \rho(H_0)$
\begin{align*}
    (H_0 -z )^{-1} - (H_V - z)^{-1} = \overline{(H_0 - z)^{-1}B^*}A(H_V - z)^{-1}.
\end{align*}
\end{theorem}

Here the similarity means that there exists an operator
$W \in \mathscr{B}(\mathcal{H})$ such that 
$W^{-1} \in \mathscr{B}(\mathcal{H})$ and 
$H_V = W H_0 W^{-1}$ holds.
Consequently, the operators~$H_V$ and~$H_0$ are isospectral.  
What is more, also the nature of the spectrum is preserved,
in particular $\sigma_\mathrm{c}(H_V) = \sigma_\mathrm{c}(H_0)$.
(The continuous spectrum $\sigma_\mathrm{c}(H)$ 
of any closed operator~$H$ in a Hilbert space~$\mathcal{H}$ 
is defined as the set of all complex points~$\lambda$
which are not eigenvalues of~$H$
such that the range of $H-\lambda$ is not equal to~$\mathcal{H}$
but its closure is.)

Instead of giving the very definition of $H_0$-smoothness 
\cite[Def.~1.2]{katoBS},
we use the equivalent criterion 
suitable to our purposes \cite[Thm.~5.1]{katoBS} 
which says that~$A$ is $H_0$-smooth if, and only if, 
\begin{equation}\label{criterion}
   \sup_{
   z \in \rho(H_0), \, \psi \in \dom(A^*) \setminus\{0\}
   } 
   \frac{\lvert \left( 
   A^*\psi, 
   \left[ R_0(z) - R_0(z^*) \right]A^*\psi
   \right) \rvert}{\lVert \psi \rVert^2} < \infty
\end{equation}
where $R_0(z) := (H_0 - z)^{-1}$.

In our case, we set $H_0 := D_0$ 
and $\mathcal{H} := L^2(\Real_+,\Com^2)$
and use the matrix polar decomposition
$V(x) = U(x) (V(x)^* V(x))^{1/2}$ with~$U(x)$ unitary
to set $A := (V^* V)^{1/4}$ and  $B^* := U (V^* V)^{1/4}$
as maximal operators of matrix multiplication in~$\mathcal{H}$.
One has $B := (V^* V)^{1/4} U^*$ 
and $\dom(A) = \dom(B) = \dom((V^* V)^{1/4})$.

We proceed with verifying the hypotheses of Theorem~\ref{kato}.
Since the domain of~$D_0$ is a subspace of 
$H^{1}(\mathbb{R}_+,\mathbb{C}^2)$, the following proposition ensures 
the right inclusion properties between domains of 
$A,B$ and~$D_0$.
\begin{proposition}
If $V \in L^1(\mathbb{R}_+,\mathbb{C}^{2,2})$,
then 
$
  H^{1}(\mathbb{R}_+,\mathbb{C}^2) 
  \subset \dom(A)
$.
\end{proposition}
\begin{proof}
Let $\psi \in C_0^\infty(\mathbb{R},\mathbb{C}^2)$ and $\eps > 0$. 
Then we have
\begin{align*}
\lvert \psi(x) \rvert^2 
=  \int_{\infty}^x {|\psi|^2}' 
= 2 \int_\infty^x \Re(\bar\psi^\top \psi') 
\leq 2 \, \|\psi\| \, \|\psi'\|  
\leq \eps \, \lVert \psi^\prime \rVert^2  
+ \eps^{-1} \, \lVert \psi \rVert^2,  
\end{align*}
for all $x \in \mathbb{R}_+$ and therefore also
$
  \|\psi\|_\infty^2 \leq \eps \, \lVert \psi^\prime \rVert^2  
+ \eps^{-1} \, \lVert \psi \rVert^2
$,
where $\|\psi\|_\infty := \||\psi|\|_{L^\infty(\Real_+)}$.
Consequently,
\begin{align*}
\lVert A \psi \rVert^2 
= \big(\psi, (V^* V)^{1/2} \psi\big) 
\leq \int_{\mathbb{R}_+} \lvert V(x) \rvert \, \lvert \psi(x) \rvert^2 \, dx 
\leq \lVert V \rVert_1 \, \lVert \psi \rVert_\infty^2 
\leq \lVert V \rVert_1 
\left( \eps \, \lVert \psi^\prime \rVert^2  + \eps^{-1} \, \lVert \psi \rVert^2 \right),
\end{align*}
where $\|V\|_1 := \||V|\|_{L^1(\Real_+)}$.
In the second inequality, we have used the equality 
of the operator norms $|(V(x)^* V(x))^{1/2}| = |V(x)|$ in~$\Com^2$. 
The density of restrictions of $C_0^\infty(\mathbb{R},\mathbb{C}^2)$ 
to~$\Real_+$ in $H^{1}(\mathbb{R}_+,\mathbb{C}^2)$ concludes the proof.
\end{proof}

To prove~\eqref{smooth} with $R_0 := (D_0-z)^{-1}$, 
we employ the fact that $K_z := A R_0(z) B^*$ 
is an integral operator in $\mathcal{H} = L^2(\Real_+,\Com^2)$ 
with the kernel (denoted by the same letter)
\begin{align*}
K_z(x,y) 
:= \lvert V(x)^* V(x) \rvert^{1/4} \
\mathcal{R}_\alpha(x,y;z) \
U(y) \, \lvert V(y)^* V(y) \rvert^{1/4} \,.
\end{align*}
The operator norm in~$\Com^2$ of the latter
can be estimated as follows:
\begin{equation}\label{criterion.estimates}
  |K_z(x,y)|
  \leq |V(x)|^{1/2} \ |\mathcal{R}_\alpha(x,y;z)| \ |V(y)|^{1/2}
  =: L_z(x,y)
  \,.
\end{equation}
The integral operator~$L_z$ in $L^2(\Real_+,\Com)$ 
generated by the kernel $L_z(x,y)$ 
obviously satisfies
$$
  \|K_z\| \leq \|L_z\|
$$
(the norms are determined by the spaces which the operators act in).

What is more, also the $D_0$-smoothness 
of~$A$ and~$B$ is reduced to the uniform boundedness of~$L_z$.
This follows from the triangle inequality applied to~\eqref{criterion}
and estimating the individual terms 
as in~\eqref{criterion.estimates}.

Finally, Lemma~\ref{supremum} yields 
$
  \|L_z\| \leq \|L\| 
$,
where~$L$ is the integral operator in $L^2(\Real_+,\Com)$  
generated by the $z$-independent kernel 
\begin{equation} 
  L(x,y) := |V(x)|^{1/2} \ 
  \sqrt{1 + \big(q + 2m \min(x,y)\big)^2} 
  \ |V(y)|^{1/2}
  \,.
\end{equation}
Applying Theorem~\ref{kato}, we therefore conclude 
with the following theorem, which is the central result of this paper.

\begin{theorem}\label{Thm.central}
Let $V \in L^1(\Real_+,\Com^{2,2})$ satisfy $\|L\| < 1$. 
Then~$D_V$ is similar to~$D_0$. 
\end{theorem}
\begin{remark}\label{Rem.sub}
The sufficient condition of Theorem~\ref{Thm.central} 
is of the same nature as the subordination requirement~\eqref{sub}.
To see it, note (\cf~\cite[Rem.~21]{KLS}) 
that~\eqref{sufcon.non} implies
that the non-relativistic variant 
$$
  L_\infty(x,y) :=  \lvert V(x) \rvert^{1/2}  
  \left(\frac{2m}{\beta} + 2m \min(x,y)\right) 
  \lvert V(y) \rvert^{1/2}  
$$
as an integral operator in $L^2(\Real_+,\Com)$ has norm less than one.
Since $L_\infty = |V|^{1/2} S_0^{-1} |V|^{1/2}$,
the last requirement is equivalent to the subordination condition
$$
  \forall \psi \in W^{1,2}(\Real_+)
  \,, \qquad
  \int_0^\infty |V| \, |\psi|^2
  \leq \frac{a}{2m} 
  \left(
  \int_0^\infty |\psi'|^2 + \beta \, |\psi(0)|^2
  \right)
  ,
$$
with $a < 1$. Note that the right-hand side is 
the quadratic form of~$S_0$ multiplied by~$a$.
We refer to~\cite[Sec.~2]{FKV} and \cite[Sec.~7.1.1]{Birman}
for the idea of the equivalence. 
Now, $L(x,y) \leq L_1(x,y) + L_2(x,y)$,
where 
$
  L_1(x,y) := \lvert V(x) \rvert^{1/2} \, \lvert V(y) \rvert^{1/2}  
$ 
is a rank-one operator and 
$$
  L_2(x,y) :=  \lvert V(x) \rvert^{1/2}  
  \, \big(q + 2m \min(x,y)\big) \,
  \lvert V(y) \rvert^{1/2}  
$$
has the same structure as~$L_\infty$. 
Consequently, $\|L\| < 1$ provided that 
$\|L_1\| \leq 1- a$ and $\|L_2\| \leq a$ with some $a \in (0,1)$.
This is equivalent to the simultaneous validity of 
$\|V\|_1 \leq 1-a$ and
$$
  \forall \psi \in W^{1,2}(\Real_+)
  \,, \qquad
  \int_0^\infty |V| \, |\psi|^2
  \leq \frac{a}{2m} 
  \left(
  \int_0^\infty |\psi'|^2 + \frac{2m}{q} \, |\psi(0)|^2
  \right)
  .
$$
\end{remark}

Estimating the operator norm of~$L$ by the Hilbert--Schmidt norm,
we get the following corollary of Theorem~\ref{Thm.central}, 
which particularly implies 
the spectral claim of Theorem~\ref{Thm.main} from the introduction.  

\begin{corollary}\label{Corol.central} 
Let $V \in L^1(\Real_+,\Com^{2,2}; (1+x) \, dx)$ 
satisfy~\eqref{sufcon}. 
Then~$D_V$ is similar to~$D_0$. 
\end{corollary}

The condition $V \in L^1(\Real_+,\Com^{2,2}; (1+x) \, dx)$ 
ensures the finiteness of the integral in~\eqref{sufcon}.
Indeed, using $\min(x,y) \leq \sqrt{xy}$ for every $x,y \in \mathbb{R}_+$,
condition~\eqref{sufcon} follows as a consequence of
\begin{equation*}
\left( \int_0^\infty \lvert V(x) \rvert \, dx \right)^2 
+ \left( \int_0^\infty \lvert V(x) \rvert \,
\big(q+2m x\big)^2 dx \right)^2
< 1 \,.
\end{equation*}

Finally, we state the following alternative result 
to Corollary~\ref{Corol.central} 
for very particularly structured potentials.
Its proof follows the same chain of reasonings as above,
just the usage of Lemma~\ref{supremum} is replaced by
Lemma~\ref{supremum.bis}.
\begin{theorem}\label{Thm.alt}
Let $V \in L^1\big((0,\infty),\Com^{2,2};(1+x) \, dx\big)$
be such that $V_{12} = V_{21} = V_{22} = 0$
and satisfy
\begin{align}\label{sufcon.alt}
 \int_0^\infty \int_0^\infty 
 \lvert V(x) \rvert \
  \max\left(1, \frac{1}{\cot(\alpha)} + 2m \min(x,y)\right)^2 \
  \lvert V(y) \rvert \ dx \, dy
 < 1 \,.
\end{align}
Then~$D_V$ is similar to~$D_0$. 
\end{theorem}
%

\section{The Dirac delta potentials and optimality}\label{Sec.optimal}
%
It is natural to ask the question whether
our sufficient condition~\eqref{sufcon} is optimal.
More specifically, by the optimality we mean that
for every $\eps > 0$ there is a potential 
$V \in L^1(\Real_+,\Com^{2,2}; (1+x) \, dx)$ 
such that 
\begin{align*} 
\lVert V \rVert^2_m :=
\int_0^\infty \int_0^\infty 
\lvert V(x) \rvert 
\left[ 1 + \big(q + 2m \min(x,y)\big)^2\right]  
\lvert V(y) \rvert \, dx \, dy = 1 + \eps \,,
\end{align*}
and~$D_V$ has an eigenvalue outside the interval 
$\sigma(D_0) = (-\infty,-m] \cup [m,+\infty)$.
 
In the case of Schr\"odinger operators 
on the half-axis~\cite{Frank-Laptev-Seiringer_2011}
as well as Dirac operators 
on the whole line~\cite{Cuenin-Laptev-Tretter_2014},
an optimality is always achieved by Dirac delta potentials.
We try to follow this approach in this paper too.

We restrict to $m > 0$.
Given $t \in \mathbb{R}$ and $a > 0$,
let~$V_t$ be formally given by 
\begin{align}\label{delta}
V_t(x) := t~\delta(x-a)
\begin{pmatrix}
1 & 0 \\
0 & 0 
\end{pmatrix}.
\end{align}
Then
%
$
\lVert V \rVert^2_m = t^2 \, [1+(q+2ma)^2]
$,
%
so the question is whether~$D_{V_t}$ has an eigenvalue in $(-m,m)$ 
whenever $t>t_*$ or $t<-t_*$ with
\begin{align*}
  t_* := \frac{1}{\sqrt{1+(q+2ma)^2}} \,.
\end{align*}
Rigorously,
we understand~$D_{V_t}$ as the self-adjoint operator which acts as~$D_0$ 
in $(0,a) \cup (a,+\infty)$ and functions 
$\psi \in H^1(\Real_+,\Com^2)$ in its domain satisfy,
in addition to the boundary condition~\eqref{bc.intro}, 
also the interface condition 
(see, \eg, \cite{lukas})
\begin{align*}
\begin{pmatrix}
 it/2 & -i \\
 1 & 0
\end{pmatrix}
\psi(a^+) +
\begin{pmatrix}
 it/2 & i \\
 -1 & 0
\end{pmatrix}
\psi(a^-) = 
\begin{pmatrix}
 0 \\
 0
\end{pmatrix}.
\end{align*}
Then we rely on the property that~$D_{V_t}$ can be approximated
by a family $D_0 + \epsilon^{-1} V(\epsilon^{-1} \cdot)$ 
with regular potentials~$V$
in a norm-resolvent sense, \cf~\cite{tusek}
(the adaptation of the proof of~\cite{tusek} 
to the present model is straightforward,
because the resolvent kernel of our operator~$D_0$ on the half-line 
is derived from its whole-line realisation~$D$).

Solving the eigenvalue problem $D_{V_t}\psi = \lambda \psi$
with $\lambda \in (-m,m)$,
we arrive at the implicit equation 
\begin{equation}\label{implicit1}  
  t = - \frac{\cot(\alpha) \sqrt{ m^2-\lambda^2 } + m -\lambda} 
  {\cot(\alpha) (m+\lambda) \sinh \left(a \sqrt{m^2-\lambda^2}\right) 
  + \sqrt{m^2-\lambda^2} \cosh \left(a \sqrt{m^2-\lambda^2}\right)}
  \, e^{a \sqrt{m^2-\lambda^2}}
  \,.
\end{equation}
Since the right-hand side is negative for any $\lambda \in (-m,m)$,
we get the expected result that 
there no solutions for all $t \geq t_0$ with some $t_0 < 0$.
On the other hand, $\lambda = 0$ is clearly a solution of~\eqref{implicit1}  
for a certain negative value of~$t$.
Hence, the critical value~$t_0$ can be chosen 
such that~$D_{V_t}$ possesses a discrete eigenvalue 
for all~$t$ less than but close to~$t_0$.
This eigenvalue emerges for $t=t_0$ either from~$m$ or~$-m$. 
Since the right-hand side of~\eqref{implicit1} diverges 
as $\lambda \to -m$, the eigenvalue must emerge from~$m$.  
Taking the limit $\lambda \to m$ in~\eqref{implicit1},
we find  
\begin{equation}
  t_0 = -  
  \frac{\cot(\alpha)}{1+2 m a \cot(\alpha)}
  \,.    
\end{equation}
The dependence of the eigenvalue on the parameter~$t$ 
is depicted in Figure~\ref{Fig.delta}.

\begin{figure}[h!]
\begin{center}
\includegraphics[width=0.9\textwidth]{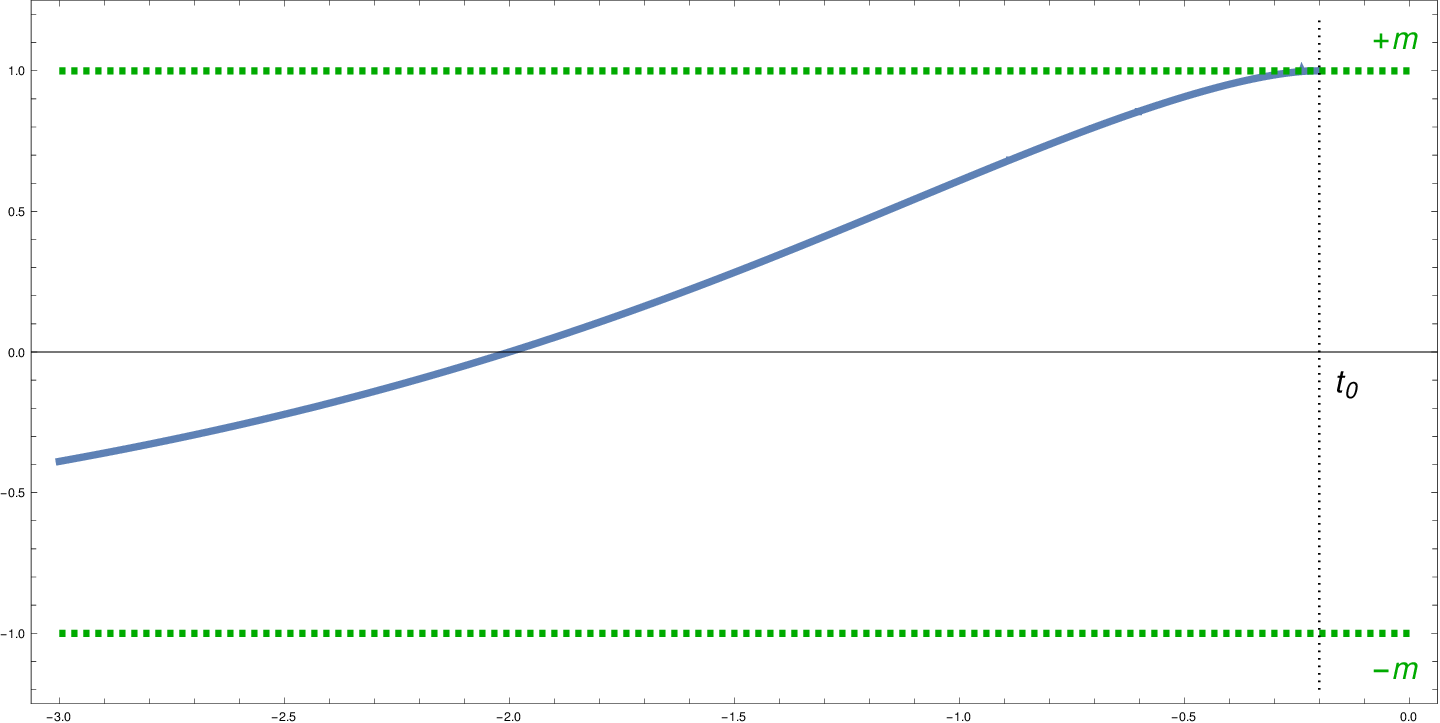}
\end{center}
\caption{The eigenvalue of $D_{V_t}$ as a function of~$t$ 
for $m=1$, $a=2$ and $\alpha=\frac{\pi}{4}$.}
\label{Fig.delta}
\end{figure}

Since $t_0 < -t_*$, we conclude that 
Theorem~\ref{Thm.main} is not optimal for the choice~\eqref{delta}.  
We also tried different matrices with different entries
or combinations being the Dirac delta interactions, 
but never achieved the optimality.
We leave as an open problem whether the optimality
of Theorem~\ref{Thm.main} can be achieved by a different
choice of the potential.

Despite failing to establish the optimality of Theorem~\ref{Thm.main},
the present analysis proves the optimality of the alternative
result stated in Theorem~\ref{Thm.alt}.
\begin{proposition}
Theorem~\ref{Thm.alt} is optimal for the choice~\eqref{delta}.
More specifically, let $\alpha \in (0,\frac{\pi}{2})$,
$m > 0$ and $a>0$ be such that
$$ 
  \frac{1}{\cot(\alpha)} + 2m a > 1 
  \,.
$$
Then there exists $\eps_0 > 0$ such that, 
given any $\eps \in (0,\eps_0)$,
the operator $D_{V_{t}}$ with $t=t_0 \sqrt{1+\eps}$
possesses a discrete eigenvalue in the interval $(-m,m)$
and the left-hand side of~\eqref{sufcon.alt} equals $1+\eps$.
\end{proposition}
%

\section{The non-relativistic limit}\label{Sec.non}
%
In order to calculate the non-relativistic limit, 
the speed of light~$c$, 
as a positive parameter to be sent to infinity, 
needs to be recovered in~$D_0$.
As for the action of~$D_0$, 
we follow the classical literature~\cite{Thaller} 
and conventionally set
$$ 
  D_{0,c} :=  
\begin{pmatrix}
mc^2 & -c\dfrac{d}{dx} \\
c\dfrac{d}{dx} & -mc^2 \\
\end{pmatrix}.
$$
Contrary to the recent approaches in
\cite{Cuenin_2014,Arrizibalaga-LeTreust-Raymond_2017,
Behrndt-Frymark-Holzmann-Stelzer-Landauer}, however,
we also make the boundary condition $c$-dependent
(and in fact also $m$-dependent):
\begin{equation*} 
\dom(D_{0,c}) := \left\lbrace \psi \in H^1(\mathbb{R}_+, \mathbb{C}^2) : \ \psi_1(0) \, \frac{\cot(\alpha)}{mc} = \psi_2(0)\right\rbrace.
\end{equation*}
This particular form is motivated by the fact that the boundary condition
can be considered as a point interaction, 
which is known to be approximated by regular potentials 
\cite{tusek}, 
and the latter require a $c$-dependence,
see~\cite[Chap.~6]{Thaller}, \cite{limit} and~\cite{lukas}.

Throughout this section we assume $m>0$.
Physically, it makes sense because there is no 
non-relativistic quantum theory for massless particles. 
Mathematically, the non-relativistic Hamiltonian~\eqref{operator.non},
which we expect to be obtained after sending~$c$ to infinity,
requires positive~$m$.

The sufficient condition~\eqref{sufcon} restated for~$D_{0,c}$
perturbed by $c$-independent potentials~$V$ reads
\begin{align*}
   \frac{1}{c^2} \int_0^\infty \int_0^\infty 
   \lvert V(x) \rvert 
   \left[1 + \big(q_c + 2mc \min(x,y)\big)^2 \right] 
   \lvert V(y) \rvert \, dx \, dy < 1 
   \,,
\end{align*} 
where  
$
  q_c := \max\left(\frac{\cot(\alpha)}{mc},\frac{mc}{\cot(\alpha)}\right)
$.
Sending the $c$ to infinity, we 
formally (here $V$ is a matrix-valued function) 
arrive at the sufficient condition~\eqref{sufcon.non}, 
which guarantees the stability of the spectrum of
the non-relativistic operator~$S_V$ 
under the identification~\eqref{identification}.   
Let us now prove that not only the sufficient 
conditions~\eqref{sufcon} and~\eqref{sufcon.non} are compatible
in the non-relativistic limit, but the operators do so as well.

\begin{theorem}
Let $\alpha \in (0,\frac{\pi}{2})$ and $m > 0$.
If $z \in \rho(S_0)$, then, for all sufficiently large~$c$, 
$z \in \rho(D_{0,c} - mc^2)$ and
\begin{align*}
\lim \limits_{c \to +\infty} 
\left\lVert 
(D_{0,c} -mc^2 - z)^{-1} 
- 
\begin{pmatrix}
(S_0 - z)^{-1} & 0 
\\
0 & 0 
\end{pmatrix}
\right\rVert = 0
\,.
\end{align*}  
\end{theorem}
\begin{proof}
In the present $c$-dependent setting,
the integral kernel of $\left(D_{0,c} - z \right)^{-1}$
reads (\cf~\eqref{resolvent-matrix form})
\begin{align*}
\mathcal{R}_{\alpha,c}(x,y;z) = \mathcal{R}_c(x,y;z) + \mathcal{R}_c(x,-y;z) \sigma_3 \eta_c(\alpha) \qquad \mbox{with} \qquad \eta_c(\alpha):=\frac{1 -i \zeta_c(z)\frac{\cot(\alpha)}{mc}}{1 + i \zeta_c(z)\frac{\cot(\alpha)}{mc}},
\end{align*}
where
$
  z \in \rho(D_{0,c}) = \Com \setminus (-\infty,-mc^2] \cup [mc^2,+\infty)
$
and
\begin{align*}
\mathcal{R}_c(x,y;z) &= \frac{1}{2c} 
\begin{pmatrix}
i\zeta_c(z) & \sgn(x-y)\\
-\sgn(x-y) & i\zeta_c^{-1}(z)
\end{pmatrix}
\exp(ik_c(z)\lvert x-y \rvert), 
\end{align*}
is the resolvent kernel of~$D_c$, 
the free $c$-dependent variant of the Dirac operator~$D$
acting on the whole line, with 
$$
  \zeta_c(z):= \frac{z+mc^2}{ck_c(z)} 
  \qquad \mbox{and} \qquad 
  k_c(z) := c^{-1} \sqrt{z^2 - (mc^2)^2}
  .
$$

On the other hand, 
 for $z \in \rho(S_0) = \mathbb{C} \setminus [0,+\infty)$, 
 the resolvent kernel of~$S_0$ is given by 
 (\cf~\cite{KLS})
\begin{align*}
\mathcal{G}_\alpha(x,y;z) := (S_0 - z)^{-1}(x,y;z) = \mathcal{G}(x,y;z) + \mathcal{G}(x,-y;z) \, \xi(\alpha),    
\end{align*}\
where
\begin{align}\label{robin resolvent}
    \xi(\alpha) := \frac{\sqrt{2mz} -i2\cot(\alpha)}{\sqrt{2mz} +i2\cot(\alpha)}
\end{align}
represents the boundary condition and
\begin{align*}
\mathcal{G}(x,y;z) := \frac{im}{\sqrt{2mz}} \exp(i\sqrt{2mz}\lvert x-y \rvert), 
\end{align*}
is the resolvent kernel of 
the whole-line Schr\"odinger operator~$S$ in $L^2(\Real,\Com)$, 
which acts as the Schr\"odinger operator~$S_0$ 
in~\eqref{operator.non}, 
subject to continuity boundary conditions at zero
up to the first derivative,
namely the domain of~$S$ is $H^2(\Real,\Com)$. 

The proof of the theorem follows by comparing the integral operators
generated by the kernels $\mathcal{R}_{\alpha,c}$ 
and $\mathcal{G}_\alpha$.
We essentially employ the well-known relativistic limit 
of~$D_c$ to~$S$ on the whole line.
To this purpose,
we denote by $\lVert \cdot \rVert_{\mathbb{R}_+}$ and $\lVert \cdot \rVert_{\mathbb{R}}$ (operator) norms on 
$L^2(\mathbb{R}_+,\mathbb{C}^2)$ 
and $L^2(\mathbb{R},\mathbb{C}^2)$, respectively. 

Given any $\phi \in L^2(\mathbb{R}_+,\mathbb{C}^2)$, 
we introduce
\begin{align*}
\phi_{\mathbb{A}_c}(x) := \phi(\lvert x \rvert)\Theta(x) + {\mathbb{A}_c}\phi(\lvert x \rvert)\Theta(-x),
\end{align*}
with ${\mathbb{A}_c} := \sigma_3 \eta_c(\alpha)$. 
Note that
\begin{align*}
\lVert \phi_{\mathbb{A}_c} \rVert^2_\mathbb{R} 
&= \int_\mathbb{R} \lvert \phi_{\mathbb{A}_c}(x) \rvert^2 dx 
= \int_{\mathbb{R}_+} \lvert \phi(x) \rvert^2 dx + \int_{\mathbb{R}_+} \lvert {\mathbb{A}_c} \phi(x) \rvert^2 dx \\
&\leq (1+\lvert \eta_c(\alpha) \rvert^2) \int_{\mathbb{R}_+} \lvert \phi(x) \rvert^2 dx = (1+\lvert \eta_c(\alpha) \rvert^2) \lVert \phi \rVert^2_{\mathbb{R}_+}.
\end{align*}
Denoting 
$
\mathbb{E}_1 := 
\begin{psmallmatrix}
1 & 0 \\
0 & 0 
\end{psmallmatrix} 
$, 
we have
\begin{align*}
\lVert (D_c -mc^2 - z)^{-1} - (S - z)^{-1} \mathbb{E}_1 \rVert_{\mathbb{R}}^2 
&= \sup \limits_{\phi \in L^2(\mathbb{R},\mathbb{C}^2)} 
\frac{\lVert \big(
(D_c -mc^2 - z)^{-1} - (S - z)^{-1}  \mathbb{E}_1 \big) \phi \rVert_{\mathbb{R}}^2}{\lVert \phi \rVert_{\mathbb{R}}^2}
\\
& \geq \sup \limits_{\phi_{\mathbb{A}_c} \in L^2(\mathbb{R},\mathbb{C}^2)} \frac{\big\lVert \big((D_c -mc^2 - z)^{-1} - (S - z)^{-1}\mathbb{E}_1\big)\phi_{\mathbb{A}_c} \big\rVert_{\mathbb{R}}^2}{\lVert \phi_{\mathbb{A}_c} \rVert_{\mathbb{R}}^2}
\\
&\geq \sup \limits_{\phi \in L^2(\mathbb{R}_+,\mathbb{C}^2)} 
\frac{\big\lVert \big( (D_{0,c} -mc^2 - z)^{-1} - (S_{\mathbb{A}_c} - z)^{-1} \mathbb{E}_1
\big)\phi \big\rVert_{\mathbb{R}_+}^2}{(1+\lvert \eta_c(\alpha) \rvert^2)\lVert \phi \rVert_{\mathbb{R}_+}^2}
\\
&= \frac{1}{(1+\lvert \eta_c(\alpha) \rvert^2)} 
\big\lVert (D_{0,c} -mc^2 - z)^{-1} - (S_{\mathbb{A}_c} - z)^{-1}\mathbb{E}_1 \big\rVert_{\mathbb{R}_+}^2 ,
\end{align*}
where  
\begin{align*}
(S_{\mathbb{A}_c} - z)^{-1}(x,y;z) 
:= \mathcal{G}(x,y;z) + \mathcal{G}(x,-y;z) \eta_c(\alpha).    
\end{align*}
Note that although $(S_{\mathbb{A}_c} - z)^{-1}$ is a well defined integral operator, it is not necessarily a resolvent of a closed, densely defined operator, symbolically denoted as $S_{\mathbb{A}_c}$. 
It is well known~\cite{Gesztesy-Grosse-Thaller_1984} that
\begin{equation}
\lVert (D_c -mc^2 - z)^{-1} - (S - z)^{-1} \mathbb{E}_1 \rVert_{\mathbb{R}}
\xrightarrow[c \to +\infty]{} 0
\,.
\end{equation}
At the same time, 
$\eta_c(\alpha) \to \xi(\alpha)$ as $c \to +\infty$.
So it remains to verify that $(S_{\mathbb{A}_c} - z)^{-1}$
is close to $(S_{0} - z)^{-1}$ in the limit $c \to +\infty$.
This follows by the Hilbert--Schmidt bound:
\begin{align*}
\lVert (S_{\mathbb{A}_c} - z)^{-1} - (S_0 - z)^{-1} \rVert_{\mathrm{HS}}^2 
&= \int_{\mathbb{R}_+} 
\int_{\mathbb{R}_+} ~\lvert \mathcal{G}(x,-y;z) \rvert^2 \,
\lvert \eta_c(\alpha) - \xi(\alpha) \rvert^2 \,  dx \, dy \\
& = \frac{m}{2z} \, \lvert \eta_c(\alpha) - \xi(\alpha) \rvert^2 \int_{\mathbb{R}_+} \int_{\mathbb{R}_+} ~\exp(-2\Im(\sqrt{2mz})\lvert x +y \rvert) \, dx \, dy
\end{align*}
together with the aforementioned convergence 
$\eta_c(\alpha) \to \xi(\alpha)$ as $c \to +\infty$.
\end{proof}
%

%
%

\subsection*{Acknowledgments}
We are grateful to Luk\'a\v{s} Heriban 
for valuable discussions on the non-relativistic limit.
The authors were supported
by the EXPRO grant No.~20-17749X 
of the Czech Science Foundation.


%
\bibliography{Dirac-reference04}
\bibliographystyle{amsplain}

\end{document}